\documentclass[letterpaper, 11pt]{article}

\usepackage{authblk}
\usepackage[english]{babel}
\usepackage[left=1 in, right=1 in, top=1.5 in, bottom=1.5 in]{geometry}
\usepackage{amsmath}
\usepackage{amssymb}
\usepackage{mathtools}
\usepackage[utf8]{inputenc}
\usepackage{amsthm}
\usepackage{tikz}
\usepackage{xcolor,hyperref}
\newtheorem{definition}{Definition}
\newtheorem{theorem}{Theorem}
\newtheorem{corollary}[theorem]{Corollary}
\newtheorem{proposition}[theorem]{Proposition}
\newcommand{\bsy}{\boldsymbol}

\newcommand{\vanish}[1]{}

\providecommand{\keywords}[1]
{
  \small	
  \textbf{{Keywords:}} #1
}

\title{Multivariate Difference Gon\v{c}arov Polynomials \\
\small{In Memory of Ron Graham} }

\author[1]{Ayomikun Adeniran\thanks{ayomikun.adeniran@pomona.edu}}  
\author[2]{Lauren Snider\thanks{lsnider@math.tamu.edu}}
\author[2]{Catherine Yan\thanks{cyan@math.tamu.edu. The third author is supported in part by Simons Collaboration Grant for Mathematics 704276. }}
\affil[1]{Department of Mathematics, Pomona College, Claremont, CA 91711}
\affil[2]{Department of Mathematics, Texas A\&M University,
College Station, TX 77843}

\date{}

\begin{document}

\maketitle

\begin{abstract} 
Univariate delta Gon\v{c}arov polynomials arise when the classical Gon\v{c}arov interpolation problem in numerical analysis is modified by replacing derivatives with delta operators. When the delta operator under consideration is the backward difference operator, we acquire the univariate difference Gon\v{c}arov polynomials, which have a combinatorial relation to lattice paths in the plane with a given right boundary. In this paper, we extend several algebraic and analytic properties of univariate difference Gon\v{c}arov polynomials to the multivariate case. We then establish a combinatorial interpretation of multivariate difference Gon\v{c}arov polynomials in terms of certain constraints on $d$-tuples of non-decreasing integer sequences. This motivates a connection between multivariate difference Gon\v{c}arov polynomials and a higher-dimensional generalized parking function, the $\bsy{U}$-parking function, from which we derive several enumerative results based on the theory of multivariate delta Gon\v{c}arov polynomials.
\end{abstract} 



\section{Introduction}

\indent The primary goal of this paper is to extend  results on univariate difference Gon\v{c}arov polynomials to multiple variables, as well as show that such polynomials have a combinatorial interpretation related to integer sequences and generalized parking functions.

Central to these problems is the theory of Gon\v{c}arov polynomials, which arose in the fields of numerical analysis and appromixation theory from an interpolation problem posed by Gon\v{c}arov \cite{goncarov}:

\vspace{2mm}
\noindent \textbf{Gon\v{c}arov Interpolation}: Find a degree $n$ polynomial $p(x)$ such that for $i = 0, 1,\ldots$, the $i$-th derivative $p^{(i)}(x)$ evaluated at a given point $a_i$ has a prescribed value $b_i$.
\vspace{2mm}

\noindent The solution to this interpolation problem consists of a linear combination of the Gon\v{c}arov polynomials $\{g_n(x;a_0, a_1, \ldots,a_{n-1})\}_{n \in \mathbb{N}}$, where $g_n(x;a_0,a_1,\ldots,a_{n-1})$ is the unique polynomial of degree $n$ satisfying the biorthogonality condition
$$ g_n^{(i)}(a_i; a_0,a_1,\ldots,a_{n-1}) = n! \delta_{in}. $$
The Gon\v{c}arov polynomials have been extensively studied for their analytical properties, but their remarkable application to parking functions, which have a vast literature in combinatorics, was surely an unforeseen consequence by Gon\v{c}arov. A (classical) parking function is a sequence $(x_1,x_2,\ldots,x_n)$ of non-negative integers whose non-decreasing rearrangement $x_{(1)} \leq x_{(2)} \leq \cdots \leq x_{(n)}$ satisfies $x_{(i)} < i$ for all $i$. The sequence $x_{(1)} \leq x_{(2)} \leq \cdots \leq x_{(n)}$ is referred to as the order statistics of $(x_1,x_2,\ldots,x_n)$. More generally, 
given a vector $\bsy{u}=(u_1, \dots, u_n)$, 
a $\boldsymbol{u}$-parking function is a sequence $(x_1,x_2,\ldots,x_n)$ of non-negative integers whose order statistics satisfy $x_{(i)} < u_i$. Kung and Yan \cite{kung} showed that Gon\v{c}arov polynomials are in direct correspondence with $\boldsymbol{u}$-parking functions, and hence the numerous algebraic and analytic properties of the former extend necessarily to the latter. 
Classical parking functions  correspond to the case $\bsy{u}=(1,2,\dots, n)$.

When the Gon\v{c}arov interpolation problem is extended to multiple variables with 
partial derivatives $\partial_{x_1}, \dots, \partial_{x_d}$, a basis of the solutions is the set of multivariate Gon\v{c}arov polynomials. Khare, Lorentz, and Yan provide a thorough treatment of bivariate Gon\v{c}arov polynomials in \cite{yan}, establishing numerous properties analogous to those of the univariate case,  and showing that a bivariate Gon\v{c}arov polynomial counts pairs of integer sequences whose order statistics satisfy certain constraints. 
This work naturally extends to $d$-dimensions and leads to 
a notion of  higher-dimensional generalized parking functions,  namely, the $\bsy{U}$-parking functions, where $\bsy{U}$ is a set of nodes in
$\mathbb{N}^d$. 

Another profound generalization of  Gon\v{c}arov polynomials is obtained by applying the rich theory of 
delta operators and finite operator calculus, which is a unified theory on linear operators analogous to the differentiation operator 
$D$ and special polynomials developed by Rota, Kahaner, and Odlyzko \cite{foundation7}.
Replacing  $D$ with an arbitrary delta operator in the Gon\v{c}arov interpolation problem, 
Lorentz, Tringali, and Yan \cite{lorentz, tringali} introduced the delta Gon\v{c}arov polynomials and 
extended many of the algebraic properties of (classical) Gon\v{c}arov polynomials to this generalized case.
They also studied multivariate delta Gon\v{c}arov polynomials  
and characterized those that are of binomial type.
A complete combinatorial interpretation for univariate delta  Gon\v{c}arov polynomials was given by Adeniran and Yan \cite{ayo} in terms of  weighted enumerators in partition lattices and exponential families.

Of particular interest to us are the difference Gon\v{c}arov polynomials, which are closely related to lattice paths and integer 
sequences.  Here the delta operator is  the backward difference operator $\Delta$. In \cite{sun}, the algebraic and combinatorial properties of the univariate difference Gon\v{c}arov polynomials are presented. 
In the current paper we seek to extend these properties to the multivariate case and investigate their combinatorial significance. 
The remainder of the paper is organized as follows. Section 2 recalls the basic definition and properties of univariate difference Gon\v{c}arov polynomials. In Section 3, we specifically examine bivariate difference Gon\v{c}arov polynomials and extend the algebraic and analytic properties of their univariate analogues to two variables.
Section 4 characterizes the relationship between bivariate difference Gon\v{c}arov polynomials and integer sequences. Finally, in Section 5 we state the corresponding results in higher dimensions.


\section{Univariate Difference Gon\v{c}arov Polynomials}

We begin by briefly summarizing the theory of delta Gon\v{c}arov polynomials with a focus on univariate difference Gon\v{c}arov polynomials. The detailed theory on delta operators is developed by Mullin and Rota in \cite{Mullin}, and the theory of delta Gon\v{c}arov polynomials is introduced in \cite{lorentz}.

Consider the vector space $\mathbb{F}[x]$ of all polynomials in the variable $x$ over a field $\mathbb{F}$ of characteristic zero. For $a \in \mathbb{F}$, let $E_a: \mathbb{F}[x]\to\mathbb{F}[x]$
 be the shift operator defined by $E_a(f) (x) = f(x+a)$, and let $\varepsilon(a): \mathbb{F}[x] \to \mathbb{F}$ be the linear functional that evaluates $p(x) \in \mathbb{F}[x]$ at $a \in \mathbb{F}$. 
A delta operator is a linear operator $\mathfrak{d}: \mathbb{F}[x] \to \mathbb{F}[x]$ that is shift-invariant, i.e., $\mathfrak{d} E_a = E_a \mathfrak{d}$ for all $a \in \mathbb{F}$, and satisfies $\mathfrak{d}(x)=c$ for some nonzero constant $c$.  
The differentiation operator $D$ is one example of a delta operator. 
Another example is the 
\textit{backward difference operator} $\Delta=I-E_{-1}$, which is defined by $\Delta p(x) = p(x) - p(x-1)$.

Every delta operator $\mathfrak{d}$ has a unique polynomial sequence $(p_n(x))_{n \in \mathbb{N}}$ such that 
$p_n(x)$ is of degree $n$,  $p_n(0)=\delta_{0n}$, and $\mathfrak{d} p_n(x) = n p_{n-1}(x)$. Such a sequence is called the basic sequence associated to $\mathfrak{d}$.   
Moreover, any shift-invariant operator $T$ can be expanded as a formal power series of $\mathfrak{d}$ by the formula
\[
T=\sum_{k \geq 0} \frac{a_k}{k!} \mathfrak{d}^k,
\]
where $a_k = \varepsilon_0(T(p_k(x))).$

For a delta operator $\mathfrak{d}$, 
suppose that $(\psi_s(\mathfrak{d}))_{s \in \mathbb{N}}$  is a sequence of linear operators of the form
$$ \psi_s(\mathfrak{d}) = \mathfrak{d}^s \sum_{r=0}^{\infty} b_{s,r} \mathfrak{d}^r, $$
where $b_{s,r} \in \mathbb{F}$ and $b_{s,0} \neq 0$. Then there exists a unique sequence of polynomials $(f_n(x))_{n \in \mathbb{N}}$ in $\mathbb{F}[x]$ such that each $f_n(x)$ has degree $n$ and satisfies
$$ \varepsilon(0) \psi_s(\mathfrak{d}) f_n(x) = n! \delta_{sn} \quad \text{for all } s \in \mathbb{N}, $$
where $\delta_{sn}$ is the Kronecker delta. In this case, we say that the polynomial sequence $(f_n(x))_{n \in \mathbb{N}}$ is \textit{biorthogonal} to the sequence of linear operators $(\psi_s(\mathfrak{d}))_{s \in \mathbb{N}}$. In fact, the polynomials $(f_n(x))_{n \in \mathbb{N}}$ form a basis of $\mathbb{F}[x]$. 

Let the delta operator $\mathfrak{d}$ be the backward difference operator $\Delta$. Then  the sequence of upper factorial functions
 $(x^{(n)})_{n\in \mathbb{N}}$ defined by $x^{(0)}=1$ and $x^{(n)}= x(x+1)\cdots(x+n-1)$ for $n \geq 1$ is the basic sequence associated to $\Delta$. 
 Given a sequence $a_0,a_1,\ldots$ of \textit{nodes} in $\mathbb{F}$, 
 let $( \psi_s(\Delta) )_{s=0}^\infty$ be the 
  sequence of linear operators given by the equation
$$ 
\psi_s(\Delta) = \Delta_s \sum_{r=0}^\infty \frac{a_s^{(r)}}{r!} \Delta^r = E_{a_s}\Delta^s. 
$$
The sequence of \emph{difference Gon\v{c}arov polynomials}  is the unique sequence of polynomials biorthogonal to 
$(\psi_s(\Delta))_{s=0}^\infty$.  That is, the $n$-th 
difference Gon\v{c}arov polynomial $\tilde{g}_n(x;a_0,a_1,\dots,a_{n-1})$ is the unique polynomial of degree $n$ satisfying 
\[
\varepsilon(a_s) \Delta^s g_n(x; a_0, \dots, a_{n-1})  =n! \delta_{s,n}, \qquad \text{ for all } s \in \mathbb{N}. 
\]
It is the difference analog of the classical univariate Gon\v{c}arov polynomial, 
which has been comprehensively studied in interpolation theory and approximation theory \cite{boas}. 

The notation for the $n$-th difference Gon\v{c}arov polynomial $\tilde{g}(x;a_0,a_1,\ldots,a_{n-1})$ reflects its dependence on only the nodes $a_0, a_1, \ldots, a_{n-1}$. The preprint \cite{sun} contains 
a set of algebraic and analytic properties for 
 $\tilde{g}_n(x;a_0,a_1,\ldots,a_{n-1})$. 
Since \cite{sun} has never been published, we include those results here for completeness. 

\begin{itemize}
\setlength\itemsep{5mm}
     \item[1.] (\textit{Determinant formula})\ 
     $ \tilde{g}_{n}(x; a_0, a_1, \dots, a_{n-1})  = n! \det M$ where $M$ is an $(n+1) \times (n+1)$ matrix whose $(i,j)$-entry, $0 \leq i,j \leq n$, is given by 
     $$ 
     m_{i,j} = \left\{ 
     \begin{array}{ll} 
      \frac{a_i^{(j-i)}}{(j-i)!}, & \text{if } 0 \leq i \leq j \text{ and } i \leq n-1 \\
      \frac{x^{(j)}}{j!}  & \text{if } i=n \\ 
      0  & \text{otherwise}. 
     \end{array} 
     \right. 
     $$
    \item[2.] (\textit{Expansion formula}) For $p(x) \in \mathbb{F}[x]$ of degree $n$,
    $$ p(x) = \sum_{i=0}^n \frac{\varepsilon(0)\psi_i(\Delta) p(x)}{i!} \tilde{g}_i(x;a_0,a_1,\ldots,a_{i-1}).$$
    \item[3.] (\textit{Linear recursion})
    $$ x^{(n)} = \sum_{i=0}^n \binom{n}{i} a_i^{(n-i)} \tilde{g}_i(x;a_0, a_1, \ldots, a_{i-1}). $$
    \item[4.] (\textit{Appell relation})
    $$ (1-t)^{-x} = \sum_{n=0}^\infty \tilde{g}_n(x;a_0,a_1,\ldots,a_{n-1}) \frac{t^n}{n!(1-t)^{a_n}}. $$
    \item[5.] (\textit{Difference relation})
    $$ \Delta\tilde{g}_n(x;a_0,a_1,\ldots,a_{n-1}) = n! \tilde{g}_{n-1}(x;a_1,a_2,\ldots,a_{n-1}) $$
    and
    $$ \tilde{g}_n(a_0; a_0,a_1,\ldots,a_{n-1}) = \delta_{0n}, $$
    which together uniquely determine the sequence of difference Gon\v{c}arov polynomials.
    \item[6.] (\textit{Shift-invariant formula})
    $$ \tilde{g}_n(x+t;a_0 + t, a_1 + t, \ldots, a_{n-1}+t) = \tilde{g}_n(x; a_0, a_1, \ldots, a_{n-1}). $$
    \item[7.] (\textit{Perturbation formula}) For  positive integers $m$ and $n$ with $m <n$,
    \begin{align*}
    &\tilde{g}_n(x; a_0, \ldots, a_{m-1}, a_m + \delta_m, a_{m+1}, \ldots, a_{n-1}) \\ &\hspace{3mm}= \tilde{g}_n(x; a_0, \ldots, a_{m-1}, a_m, a_{m+1}, \ldots, a_{n-1}) \\ &\hspace{5mm}- \binom{n}{m} \tilde{g}_{n-m}(a_m + \delta_m; a_m, a_{m+1}, \ldots, a_{n-1}) \tilde{g}_m(x;a_0,a_1,\ldots,a_{m-1}).
    \end{align*}
    \item[8.] (\textit{Sheffer relation})
    \[
    \tilde{g}_n(x+y; a_0, a_1, \ldots, a_{n-1}) =\sum_{i=0}^n \binom{n}{i} \tilde{g}_{n-i}(y; a_i,\ldots, a_{n-1}) x^{(i)}.
    \]
    In particular, letting $y=0$ we obtain the expansion of $\tilde{g}_n(x; a_0, a_1, \ldots, a_{n-1})$ under the basis 
    $( x^{(n)})_{n \in \mathbb{N} }$. 
\end{itemize}

Difference Gon\v{c}arov polynomials are useful in combinatorics due to their connection with lattice paths in the plane with a given right boundary. Let $x,n$ be positive integers. A lattice path in $\mathbb{Z}^2$ from $(0,0)$ to $(x-1,n)$ with steps $(1,0)$ and $(0,1)$ can be recorded by a
non-decreasing integer sequence $(x_0, x_1, \dots, x_{n-1})$, where $(x_i,i)$ is the coordinate of the rightmost point on the lattice path and the line $y=i$.  Given $a_0 \leq a_1 \leq \cdots \leq a_{n-1} \in [0, x]^n$, let $LP_n(a_0, a_1, \dots, a_{n-1})$ be the number of lattice paths $(x_0, x_1, \dots,x_{n-1})$ from $(0,0)$ 
to $(x-1, n)$ such that $0\leq x_i < a_i$ for  $0\leq i \leq n$. 
Then we have the following theorem. 

\begin{theorem}\cite[Theorem 4.1]{sun} 
\begin{eqnarray*}
LP_n(a_0, a_1, \dots, a_{n-1}) & = &  \frac{1}{n!} \tilde{g}_n(x; x-a_0, x-a_1, \dots, x-a_{n-1}) \\
& =  & \frac{1}{n!} \tilde{g}_n (0; -a_0, -a_1, \dots, -a_{n-1}). 
\end{eqnarray*} 
\end{theorem} 

When $a_i = a$ for all $i$, $\tilde{g}_n(x; a, \dots, a) = (x-a)^{(n)}$. Hence $LP_n(a, \dots, a)= \frac{a^{(n)}}{n!} = \binom{a+n-1}{n}$, which is clearly the number of lattice paths from $(0,0)$ to $(a-1, n)$. 
When $a_i=a+(i-1)b$, $\tilde{g}_n(x; a, a+b, \dots, a+(n-1)b)
= (x-a)(x-a-nb+1)^{(n-1)}$ for $n > 0$. In particular, for
$a=b=1$, $\tilde{g}_n(0; -a_0, -a_1,\dots, -a_{n-1})$ is the Catalan number $\frac{1}{n+1} \binom{2n}{n}$;  when $a=1$, 
$b \in \mathbb{N}$, we get the Fuss-Catalan number  
$\frac{1}{bn+1} \binom{ (b+1)n}{n}$. For general values of 
$a_i$'s, $\tilde{g}_n(0; a_0, \dots, a_{n-1})$ can be computed by the determinant formula or the linear recursion. 

Lattice paths are a classical subject of study in combinatorics, having a vast literature with applications in many fields of mathematics, computer science, physics and statistics.
For the  combinatorial theory of lattice paths, 
see the monograph \cite{Mohanty} by Mohanty and the more recent comprehensive survey \cite{Kratten} by Krattenthaler. 
In addition of being a basic but useful tool in lattice path counting,
difference Gon\v{c}arov polynomials provide a new perspective to lattice paths and 
connect them to other combinatorial structures that are associated with general delta operators. The most notable examples are various generalization of parking functions, which are the combinatorial structures associated with the differential operator. In fact, 
difference Gon\v{c}arov polynomials have already appeared in enumerating parking distributions over a caterpillar graph \cite{graham}, and in enumerating increasing parking sequences \cite{ayo-yan}. 


\section{Bivariate Difference Gon\v{c}arov Polynomials}
By replacing the difference operator $\Delta$ with a set of difference operators $\{\Delta_{x_i}\}_{i=1}^d$, where 
$d$ is a positive integer, we can define a system of multivariate biorthogonal polynomials in $\mathbb{F}[x_1, \dots, x_d]$  that naturally extend the univariate difference Gon\v{c}arov polynomials to multiple variables. A general theory of systems of delta operators and delta Gon\v{c}arov polynomials in multi-variables was introduced in \cite{tringali} .   In this paper we only need a special case: the system of delta operators is $(\Delta_{x_1}, \Delta_{x_2}, \dots, \Delta_{x_d})$, where $\Delta_{x_i}$ is the backward difference operator with respect to the variable $x_i$.  
We will first state the definition and the basic properties from the general theory established in \cite{tringali}. Then we 
present some special algebraic properties of  multivariate difference Gon\v{c}arov polynomials. In the next section we 
discuss the combinatorial significance of such multivariate polynomials. 

For simplicity and clarity,
in Sections 3 and 4 we  restrict our attention to the  bivariate case. 
All the results can be  extended easily to the multivariate cases, which we describe briefly in Section 5.


Fix positive integers $m$ and $n$. We write $(i,j) \preceq (m,n)$ if $ i \leq m$ and $ j \leq n$. Let $S_{m,n}$ denote the poset $\{(i,j): (0,0) \preceq (i,j) \preceq (m,n)\}$ and denote the space of all bivariate polynomials having coordinate degree $(m,n)$ by $\Pi^2_{m,n}$. That is,
$$ \Pi^2_{m,n} = \{ \sum_{(i,j) \in S_{m,n}} b_{i,j} x^i y^j : b_{i,j} \in \mathbb{F} \}.$$


\indent The following is a bivariate variation of the Gon\v{c}arov interpolation problem, with difference operators replacing 
differential operators.\\

\noindent \textbf{Bivariate Gon\v{c}arov Interpolation with Difference Operators.}
Fix a node-set $\bsy{Z} = \{z_{i,j} = (x_{i,j},y_{i,j}) : (i,j) \in S_{m,n} \}$.
Given a set of  numbers $\{b_{i,j} \in \mathbb{F} : (i,j) \in S_{m,n}\}$, find a polynomial $p(x,y) \in \Pi^2_{m,n}$ such that, for all $(i,j) \in S_{m,n}$,
$$ \varepsilon(z_{i,j}) \Delta_x^i \Delta_y^j p(x,y) = b_{i,j}. $$

 
From the general theory developed in \cite{tringali}, we have 
that for any values $\{ b_{i,j}: (i,j) \in S_{m,n}\}$, the bivariate Gon\v{c}arov interpolation problem with difference operators  has a unique solution in the space $\Pi^2_{m,n}$.  
In particular, by taking all but one of $\{b_{i,j}: (i,j) \in S_{m,n}\}$ to be $0$,   we can define the bivariate difference Gon\v{c}arov polynomials. 

\begin{definition}
Let  $\bsy{Z} = \{z_{i,j} = (x_{i,j},y_{i,j}) : (i,j) \in S_{m,n} \}$  be a set of nodes. The bivariate difference Gon\v{c}arov polynomial $\tilde{g}_{m,n}((x,y);\bsy{Z})$ is the unique polynomial in $\Pi^2_{m,n}$ satisfying
\begin{eqnarray}  \label{definition} 
\varepsilon(z_{i,j}) \Delta_x^i \Delta_y^j \tilde{g}_{m,n}((x,y);\bsy{Z}) = m! n! \delta_{m,i} \delta_{n,j}
\end{eqnarray} 
for all $(i,j) \in S_{m,n}$.
\end{definition}

It follows that $\tilde{g}_{0,0}((x,y);\bsy{Z})=1$. 
For the special grid $\bsy{O}=\{z_{i,j}=(0,0): (i,j) \in S_{m,n}\}$, the set of  difference Gon\v{c}arov polynomials $\{\tilde{g}_{m,n}((x,y);\bsy{O})\}_{m,n \in \mathbb{N}}$ is called the basic sequence of the system $(\Delta_x, \Delta_y)$. 
From the interpolation conditions \eqref{definition}, it is easy to check that 
\[
\tilde{g}_{m,n}((x,y);\bsy{O}) = x^{(m)} y^{(n)}.
\]

In general, the set $\{ \tilde{g}_{i,j}((x,y);\bsy{Z}):  (i,j) \in S_{m,n}\}  $ forms a basis to the solutions of the bivariate Gon\v{c}arov interpolation problem with difference operators. 
Next we discuss the algebraic properties of bivariate difference Gon\v{c}arov polynomials, analogous to those of the univariate case. 
We remark that Theorems \ref{Expansion}, \ref{Linear}, and 
\ref{Abel-thm} are special cases of  Propositions 3.5, 3.6 and Theorem 5.1 in \cite{tringali}, 
and the other results are new. 

\begin{theorem}{(Expansion formula)}\label{Expansion} 
For any $p(x,y) \in \Pi^2_{m,n}$,
$$ p(x,y) = \sum_{i=0}^m \sum_{j=0}^n \frac{1}{i!j!} \bigg[ \varepsilon(z_{i,j}) \Delta_x^i \Delta_y^j p(x,y) \bigg] \tilde{g}_{i,j}((x,y);\bsy{Z}). $$
\end{theorem}
\noindent \textit{Proof.} This property follows immediately  from the definition of bivariate difference Gon\v{c}arov polynomials and the fact that $\{\tilde{g}_{i,j}((x,y);\bsy{Z})\}_{(i,j) \preceq  (m,n)}$ forms a basis of $\Pi^2_{m,n}$.   \hfill \qedsymbol

\begin{theorem}{(Linear recursion)} \label{Linear} 
\begin{eqnarray} \label{linear} 
 x^{(m)} y^{(n)} = \sum_{i = 0}^m \sum_{j = 0}^n 
 \binom{m}{i}\binom{n}{j}
 x_{i,j}^{(m-i)} y_{i,j}^{(n-j)} \tilde{g}_{i,j}((x,y);\bsy{Z})
 \end{eqnarray} 
\end{theorem}
\noindent \textit{Proof.} It is obtained by letting $p(x,y) = x^{(m)}y^{(n)}$ in the expansion formula. \hfill \qedsymbol

\begin{theorem}{(Appell relation)} \label{Appell} 
$$ (1-s)^{-x}(1-t)^{-y} = \sum_{m=0}^{\infty} \sum_{n=0}^{\infty} \tilde{g}_{m,n}((x,y);\bsy{Z}) \frac{s^m}{(1-s)^{x_{m,n}} m!} \frac{t^n}{(1-t)^{y_{m,n}} n!} $$
\end{theorem}
\noindent \textit{Proof.} Using Taylor expansion and the linear recursion formula, we have
\begin{align*}
\frac{1}{(1-s)^x(1-t)^y} &= \sum_{m=0}^\infty \sum_{n=0}^\infty \frac{x^{(m)}s^m}{m!}\frac{y^{(n)}t^n}{n!} \\
				&= \sum_{m=0}^\infty \sum_{n=0}^\infty \frac{s^m t^n}{m!n!} \sum_{i=0}^m \sum_{j=0}^n \binom{m}{i} \binom{n}{j}  x_{i,j}^{(m-i)} y_{i,j}^{(n-j)} \tilde{g}_{i,j}((x,y);\bsy{Z}) \\
				&= \sum_{i=0}^\infty \sum_{m=i}^\infty \sum_{j=0}^\infty \sum_{n=j}^\infty \frac{1}{m!} \binom{m}{i} \frac{1}{n!}
				\binom{n}{j} x_{i,j}^{(m-i)} s^m y_{i,j}^{(n-j)} t^n \tilde{g}_{i,j}((x,y);\bsy{Z}) \\
				&= \sum_{i=0}^\infty \sum_{j=0}^\infty \tilde{g}_{i,j}((x,y);\bsy{Z}) \sum_{m=i}^\infty \frac{1}{m!} \binom{m}{i} x_{i,j}^{(m-i)} s^m \sum_{n=j}^\infty \frac{1}{n!} \binom{n}{j} y_{i,j}^{(n-j)} t^n \\
				&= \sum_{i=0}^\infty \sum_{j=0}^\infty \tilde{g}_{i,j}((x,y);\bsy{Z}) \frac{s^i}{i!}
				\frac{t^j}{j!} 
				\sum_{m=i}^\infty \frac{x_{i,j}^{(m-i)}s^{m-i}}{(m-i)!}  \sum_{n=j}^\infty \frac{y_{i,j}^{(n-j)}t^{n-j}}{(n-j)!} \\
				&= \sum_{i=0}^\infty \sum_{j=0}^\infty  \tilde{g}_{i,j}((x,y);\bsy{Z}) \frac{s^i}{i!}
				\frac{t^j}{j!} 
				\sum_{m=0}^\infty \frac{x_{i,j}^{(m)}s^{m}}{m!}  \sum_{n=0}^\infty \frac{y_{i,j}^{(n)}t^{n}}{n!} \\
                & = \sum_{i=0}^\infty \sum_{j=0}^\infty  \tilde{g}_{i,j}((x,y);\bsy{Z}) \frac{s^i}{i!}
				\frac{t^j}{j!} 				
				\frac{1}{(1-s)^{x_{i,j}}}\frac{1}{(1-t)^{y_{i,j}}}.
\end{align*}
\noindent This completes the proof. \hfill \qedsymbol \\

\indent The following two formulas are analogues of the differential and integral relations of the classical bivariate Gon\v{c}arov polynomials studied in \cite{yan}.  For a node-set $\bsy{Z} = \{z_{i,j} : i,j \in \mathbb{N}\}$, let $\bsy{LZ} = \{w_{i,j} : w_{i,j} = z_{i+1,j}, i,j \in \mathbb{N}\}$ and $\bsy{DZ} = \{w_{i,j} : w_{i,j} = z_{i,j+1}, i,j \in \mathbb{N}\}$. From here on we will assume $\bsy{Z}$ is an infinite grid with indices $i,j \in \mathbb{N}$, and 
$\tilde{g}_{m,n}((x,y);\bsy{Z})$ is determined by the subset  
$\{ z_{i,j} \in \bsy{Z}: (i,j) \in S_{m,n}\}$. 

\begin{theorem}{(Difference relations)}
For any $m, n \in \mathbb{N}$, 
$$ \Delta_x \tilde{g}_{m,n}((x,y);\bsy{Z}) = m \tilde{g}_{m-1,n} ((x,y);\bsy{LZ}),$$ 
$$\Delta_y \tilde{g}_{m,n}((x,y);\bsy{Z}) = n \tilde{g}_{m,n-1} ((x,y);\bsy{DZ}).$$
\end{theorem}
\noindent \textit{Proof.} We will only prove the first relation, as the second follows by symmetry. We wish to show that $\Delta_x \tilde{g}_{m,n}((x,y);\bsy{Z})$ and $m \tilde{g}_{m-1,n}((x,y);\bsy{LZ})$ satisfy the same biorthogonality conditions. Now the definition of $\tilde{g}_{m,n}((x,y);\bsy{LZ})$ implies that
$$ \varepsilon(z_{i+1,j}) \Delta_x^i \Delta_y^j \bigg[ \Delta_x \tilde{g}_{m,n}((x,y);\bsy{Z}) \bigg] = \varepsilon(z_{i+1,j}) \Delta_x^{i+1} \Delta_y^j \tilde{g}_{m,n}((x,y);\bsy{Z}) = 0$$
when $(i,j) \preceq (m-1,n)$ with $(i,j) \neq (m-1,n)$. When $(i,j) = (m-1,n)$,
$$ \varepsilon(z_{i+1,j}) \Delta_x^{i+1} \Delta_y^j \tilde{g}_{m,n}((x,y);\bsy{Z}) = \varepsilon(z_{m,n}) \Delta_x^m \Delta_y^n \tilde{g}_{m,n}((x,y);\bsy{Z}) = m!n!.$$
Since $m \tilde{g}_{m-1,n}((x,y);\bsy{LZ})$ satisfies these same conditions, uniqueness of the interpolation yields the first difference relation. \hfill \qedsymbol

\begin{corollary} 
The general difference formula is 
$$
\Delta_x^i \Delta_y^j \tilde{g}_{m,n}((x,y); \bsy{Z}) = (m)_i (n)_j \tilde{g}_{m-i, n-j}((x,y); \bsy{L}^i\bsy{D}^j \bsy{Z}), 
$$
where $(t)_k = t(t-1)\cdots (t-k+1)$ is the $k$-th lower factorial of $t$. 
\end{corollary} 

\vanish{
\begin{theorem}{(Summation formula)}
$$\tilde{g}_{m,n}((x,y);\bsy{Z}) = \sum_{a = x_{0,0}+1}^x \sum_{b = y_{0,0}+1}^y mn \tilde{g}_{m-1,n-1}((a,b);\bsy{LDZ}) + \tilde{g}_{m,n}((x_{0,0},y);\bsy{Z}) + \tilde{g}_{m,n}((x,y_{0,0});\bsy{Z})$$
\end{theorem}
} 

\begin{theorem}{(Shift-invariant formula)} \label{shift} 
Given a node-set $\bsy{Z} = \{z_{i,j}=(x_{i,j},y_{i,j}) : i,j \in \mathbb{N}\}$, let $\bsy{Z} + (\xi,\eta)$ denote the set $\{(x_{i,j} + \xi, y_{i,j} + \eta) : i,j \in \mathbb{N}\}$. Then we have
$$ \tilde{g}_{m,n}((x + \xi, y + \eta); \bsy{Z} + (\xi,\eta)) = \tilde{g}_{m,n}((x,y);\bsy{Z}). $$
\end{theorem}
\noindent \textit{Proof.} By definition, $\tilde{g}_{m,n}((x,y); \bsy{Z} +(\xi,\eta))$ is the unique polynomial in $\Pi^2_{m,n}$ satisfying interpolation conditions
$$ \varepsilon(z_{i,j})E_x^\xi E_y^\eta \Delta_x^i \Delta_y^j \tilde{g}_{m,n}((x,y); \bsy{Z} + (\xi,\eta)) = m! n! \delta_{im} \delta_{jn} $$
for $(i,j) \preceq (m,n)$, where $E_x^a$ and $E_y^b$ are the shift operators $(E_x^a f)(x,y) = f(x+a,y)$ and $(E_y^b f)(x,y) = f(x,y+b)$, respectively. Since these shift operators commute with the difference operators $\Delta_x$ and $\Delta_y$, we may equivalently express the interpolation conditions as
$$ \varepsilon(z_{i,j}) \Delta_x^i \Delta_y^j \tilde{g}_{m,n}((x + \xi, y + \eta); \bsy{Z} + (\xi,\eta)) = m! n! \delta_{im} \delta_{jn}, $$
which are the precise conditions satisfied by $\tilde{g}_{m,n}((x,y);\bsy{Z})$. \hfill \qedsymbol

\begin{theorem}{(Perturbation formula)} \label{Perturbation} 
Given a set of nodes $\bsy{Z}$, suppose we perturb the $(i_0,j_0)$-th node of $\bsy{Z}$ to $z_{i_0,j_0}^*$. Let $\bsy{Z}^*$ be the new set of nodes. Then for $(i_0,j_0) \preceq (m,n)$ but 
$(i_0, j_0) \neq (m,n)$, we have
$$ \tilde{g}_{m,n}((x,y);\bsy{Z}^*) = \tilde{g}_{m,n}((x,y);\bsy{Z}) - \binom{m}{i_0} \binom{n}{j_0}  \tilde{g}_{m-i_0,n-j_0}(z_{i_0,j_0}^*;\bsy{L}^{i_0}\bsy{D}^{j_0}\bsy{Z}) \tilde{g}_{i_0,j_0}((x,y);\bsy{Z}). $$
\end{theorem}
\noindent \textit{Proof.} Let
$$h_{m,n}(x,y) = \tilde{g}_{m,n}((x,y);\bsy{Z}^*) + \frac{1}{i_0! j_0!} \bigg[ \varepsilon(z_{i_0,j_0}^*) \Delta_x^{i_0} \Delta_y^{j_0} \tilde{g}_{m,n}((x,y);\bsy{Z}) \bigg] \tilde{g}_{i_0,j_0}((x,y);\bsy{Z}).$$
One can easily check that  $h_{m,n}(x,y)$ and $\tilde{g}_{m,n}((x,y);\bsy{Z})$ satisfy the same interpolation conditions and so are equal by uniqueness. Using the  difference relations, we may rewrite
\begin{align*}
h_{m,n}(x,y) &= \tilde{g}_{m,n}((x,y);\bsy{Z}^*) + \binom{m}{i_0} \bigg[ \varepsilon(z_{i_0,j_0}^*) \frac{1}{j_0!} \Delta_y^{j_0} \tilde{g}_{m-i_0,n}((x,y);\bsy{L}^{i_0}\bsy{Z}) \bigg] \tilde{g}_{i_0,j_0}((x,y);\bsy{Z}) \\
			&= \tilde{g}_{m,n}((x,y);\bsy{Z}^*) + \binom{m}{ i_0}
			\binom{n}{j_0} \tilde{g}_{m-i_0,n-j_0}(z_{i_0,j_0}^*;\bsy{L}^{i_0}\bsy{D}^{j_0}\bsy{Z}) \tilde{g}_{i_0,j_0}((x,y);\bsy{Z}),
\end{align*}
and the statement is proved. \hfill \qedsymbol

\textsc{Remark}. In Theorem \ref{Perturbation} if $(i_0,j_0)=(m,n)$, then $\tilde{g}_{m,n}((x,y);\bsy{Z}^*) = \tilde{g}_{m,n}((x,y);\bsy{Z})$. This is because the difference operators are degree-reducing, in the sense that for a polynomial $p(x,y)$ of coordinate degree $(a,b)$, 
$\Delta_x p(x,y)$ is of degree $(a-1, b)$ and $\Delta_y p(x,y)$ is of degree $(a, b-1)$. Hence $\Delta_x^m \Delta_y^n 
\tilde{g}_{m,n}((x,y);\bsy{Z})$ is always a constant, which must 
be equal to $m!n!$ by the interpolation conditions. It also implies that the formula of $\tilde{g}_{m,n}((x,y); \bsy{Z})$ does not depend  on the node $z_{m,n}$.

\begin{theorem}{(Sheffer Relation)} \label{Sheffer} 
\[
\tilde{g}_{m,n}((x+b, y+c); \bsy{Z}) = \sum_{i=0}^m \sum_{j=0}^n \binom{m}{i} \binom{n}{j} \tilde{g}_{m-i, n-j}((b,c); \bsy{L}^i \bsy{D}^j\bsy{Z}) x^{(i)} y^{(j)}. 
\]
\end{theorem} 
\noindent \textit{Proof.}
Expanding the polynomial $\tilde{g}_{m,n}((x+b, y+c); \bsy{Z})$
under the basis $\{\tilde{g}_{i,j}((x,y); \bsy{O}) \}_{(i,j) \preceq (m,n)}$ by Theorem \ref{Expansion} and noting that 
$\tilde{g}_{i,j}((x,y);\bsy{O}) = x^{(i)} y^{(j)}$,
we have 
\begin{eqnarray*} 
\tilde{g}_{m,n}((x+b, y+c); \bsy{Z})
& =  &\sum_{i=0}^m \sum_{j=0}^n \frac{1}{i!j!} \bigg[ \varepsilon(0) \Delta_x^i \Delta_y^j \tilde{g}_{m,n}((x+b,y+c);\bsy{Z}) \bigg] 
 x^{(i)}y^{(j)} \\ 
& =  & \sum_{i=0}^m \sum_{j=0}^n \frac{1}{i!j!} \bigg[
 \varepsilon(0) (m)_i (n)_j 
 \tilde{g}_{m-i, n-j}( (x+b, y+c); \bsy{L}^i \bsy{D}^j \bsy{Z}) \bigg]
  x^{(i)}y^{(j)} \\  
  & = & \sum_{i=0}^m \sum_{j=0}^n \binom{m}{i} \binom{n}{j} \tilde{g}_{m-i, n-j}((b,c); \bsy{L}^i \bsy{D}^j\bsy{Z}) x^{(i)} y^{(j)}. 
\end{eqnarray*} 
\hfill \qedsymbol

The following observations give a relation between the univariate and bivariate  difference Gon\v{c}arov polynomials. 
Both can be checked easily using Definition 1. 
\begin{enumerate}
    \item When $m=0$ or $n=0$, we have 
    \begin{eqnarray*} 
     \tilde g_{m,0}((x,y);\bsy{Z}) &= &\tilde g_m (x; x_{0,0}, x_{1,0}, \dots, x_{m-1,0}), \\ 
     \tilde g_{0,n}((x,y);\bsy{Z})& = &\tilde g_n(y; y_{0,0}, y_{0,1}, \dots, y_{0, n-1}).
    \end{eqnarray*} 
    \item If there exist some sequences $\{\alpha_i\}$ and $\{\beta_j\}$ such that
    $z_{i,j}=(x_{i,j}, y_{i,j}) =(\alpha_i, \beta_j)$, then 
    \[
    \tilde g_{m,n}((x,y);\bsy{Z}) = \tilde{g}_m(x; \alpha_0,\dots, \alpha_{m-1}) \tilde g_{n}(y; \beta_0, \dots, \beta_{n-1})
    \]
    is the product of univariate difference Gon\v{c}arov polynomials. 
\end{enumerate}

In general, the closed formula of $\tilde g_{m,n}((x,y);\bsy{Z})$ is quite involved. However, a special case in which we have an elegant 
closed formula of $\tilde g_{m,n}((x,y);\bsy{Z})$ occurs
when the node $z_{i,j}$ is 
a linear transformation of $(i,j)$. The resulting polynomials
$\{ \tilde g_{m,n}((x,y);\bsy{Z}): m,n\in \mathbb{N}\}$ are
 called \emph{delta Abel polynomials}, since they are analogs of the Abel polynomial $A(x)=x(x-a)^{n-1}$ and satisfy a  multivariate identity of binomial type.

\begin{theorem}  \label{Abel-thm} 
Assume that $\bsy{Z}$ is a linear transformation of $\mathbb{N}^2$ by a 
$2 \times 2$ matrix $A$, i.e., there are constants $a, b, c, d$ such that $x_{i,j} = a i+bj$ and  $y_{i,j}=ci+dj$ 
for all $i, j \in \mathbb{N}$. Then 
\begin{eqnarray} \label{Abel-2}
\tilde g_{m,n}(x,y;\bsy{Z}) = (xy- x_{0,n} y - y_{m,0} x ) (x-x_{m,n}+1)^{(m-1)} (y-y_{m,n}+1)^{(n-1)}. 
\end{eqnarray} 
\end{theorem}
\noindent \textit{Proof.}
It follows from 
Theorem 5.1 of \cite{tringali} and the fact that $(x^{(n)})_{n \in \mathbb{N}}$ is the basic sequence of the delta operator $\Delta_x$. 
\hfill \qedsymbol

\vanish{
\noindent \textsc{Remark}. Let $d$ be a fixed integer $\geq 1$. For a vector $\mathbf{v} \in \mathbb{F}^d$, we denote by $v_j$ the $j$-th component of $\mathbf{v}$. Given $\mathbf{n}=(n_1, \dots, n_d) \in \mathbb{N}^d$, we set $\mathbf{n}!=n_1!n_2!\cdots n_d!$. For $\mathbf{k}, \mathbf{n} \in \mathbb{N}^d$, $
\mathbf{k} \preceq \mathbf{n}$ means $k_i \leq n_i$ for all $i$, and $\binom{\mathbf{n}}{\mathbf{k}} = \binom{n_1}{k_1} \cdots \binom{n_d}{k_d}$. With such notations, we can define the $d$-dimensional difference Gon\v{c}arov polynomials 
with respect to the system of difference operators $(\Delta_{x_1}, \dots, \Delta_{x_d})$. Given a grid 
$Z=\{ z_{\mathbf{k}} \in \mathbb{F}^d: \mathbf{k} \in \mathbb{N}^d\}$, there is a unique polynomial $t_{\mathbf{n}}(\mathbf{x};Z)$ of coordinate degree $\mathbf{n} \in \mathbb{N}^d$ satisfying 
\[
\varepsilon(z_{\mathbf{k}}) \Delta_{x_1}^{k_1} \cdots 
\Delta_{x_d}^{k_d} (t_{\mathbf{n}}(\mathbf{x}; Z)) = \mathbf{n}! \delta_{\mathbf{k}, \mathbf{n}} 
\]
for all $\mathbf{k} \preceq \mathbf{n}$. This polynomial 
is the multivariate difference Gon\v{c}arov polynomial indexed by $\mathbf{n}$.  Theorems 2--9 can all be extended  straightforwardly to $d$-dimensions, where the summations are over the set  $\{\mathbf{k} \in \mathbb{N}^d: \mathbf{k} \preceq \mathbf{n}\}$, and the binomial coefficient $\binom{m}{i} \binom{n}{j}$ is replaced with $\binom{\mathbf{n}}{\mathbf{k}}$.  
}


\section{Bivariate Difference Gon\v{c}arov Polynomials and Integer Sequences}

In this section  we focus on the combinatorial significance of  bivariate difference Gon\v{c}arov polynomials.  
Just as the univariate difference Gon\v{c}arov polynomials describe lattice paths with a right boundary, or equivalently non-decreasing integer sequences with an upper bound, the bivariate difference Gon\v{c}arov polynomials capture the structure of a pair of  non-decreasing integer sequences whose joint distribution is bounded by a set of constraints. First we  introduce the combinatorial model and the necessary notations.  

\indent Let $m,n \in \mathbb{N}$, and suppose $ \bsy{U}$ is a set of weight-vectors
$$ \bsy{U} = \{(u_{i,j},v_{i,j})\in \mathbb{N}^2: i,j \in \mathbb{N}, \ u_{i,j} \leq u_{i',j'}, \ v_{i,j} \leq v_{i',j'} \text{ whenever } (i,j) \preceq (i',j')  \}.
$$
Define $D_{m,n}$ to be the directed graph having as vertices the points $\{(i,j) : 0 \leq i \leq m, 0 \leq j \leq n\}$ and having as edges all north steps $N=(0, 1)$ and east steps $E = (1,0)$ connecting its vertices. Assign every edge $e$ of $D_{m,n}$ a weight $wt(e)$  by
letting 
\[ wt(e) = \begin{cases} 
      u_{i,j} & \textnormal{if } e \textnormal{ is an east step from } (i,j) \textnormal{ to } (i+1,j), \\
      v_{i,j} & \textnormal{if } e \textnormal{ is a north step from } (i,j) \textnormal{ to } (i,j+1).
   \end{cases}
\]

\indent Given a lattice path $P$ from the origin $O=(0,0)$ to the point $(m,n)$, we write $P = e_1 e_2 \ldots e_{m+n}$, where $e_i \in \{E,N\}$, to record the sequence of steps of $P$. Thus, $P$ must have exactly $m$ E-steps and $n$ N-steps. Consider a pair of non-decreasing integer sequences $(\bsy{a},\bsy{b})$ with $\bsy{a} = (a_0, a_1, \ldots, a_{m-1})$ and $\bsy{b} = (b_0, b_1, \ldots, b_{n-1})$. We say that \emph{the pair
$(\bsy{a},\bsy{b})$ is bounded by $P$ with respect to the set $\bsy{U}$}  if and only if, for $r = 1, 2, \ldots, m+n$,
\[ \begin{cases} 
      a_{i} < u_{i,j} & \textnormal{if } e_r \textnormal{ is an E-step from } (i,j) \textnormal{ to } (i+1,j), \\
      b_{j} < v_{i,j} & \textnormal{if } e_r \textnormal{ is a N-step from } (i,j) \textnormal{ to } (i,j+1).
   \end{cases}
\]

\bigskip
\noindent \textbf{Example.} Let $\bsy{U} = \{(u_{i,j},v_{i,j}) : 0 \leq i \leq 3, \ 0 \leq j \leq 4\}$ be given by $u_{i,j} = j+1$ and $v_{i,j} = i+1$. The pair $(\bsy{a},\bsy{b})$ with $\bsy{a} = (2, 2, 3)$ and $\bsy{b} = (0, 0, 1, 3)$ is bounded by the lattice path $P = NNENEEN$ in bold in the figure below.  Note that the lattice path bounding $(\bsy{a},\bsy{b})$  may not be unique. For example, $P'=NNEENEN$ is another such path. 

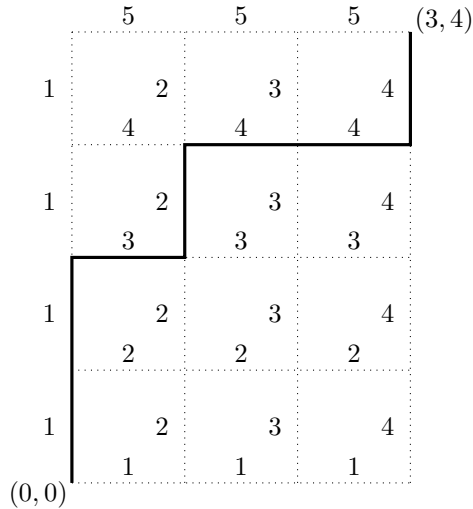
\begin{center}
\begin{tikzpicture}[scale=1.5]
\draw[step=1cm,dotted] (0,0) grid (3,4);
\draw (-.3,-.1) node {$(0,0)$};
\draw (3.3,4.1) node  {$(3,4)$};
\draw (.5,.15) node  {$1$};  
\draw (1.5,.15) node {$1$};
\draw (2.5,.15) node {$1$};
\draw (.5,1.15) node {$2$};
\draw (1.5,1.15) node {$2$};
\draw (2.5,1.15) node {$2$};
\draw (.5,2.15) node   {$3$};
\draw (1.5,2.15) node {$3$};
\draw (2.5,2.15) node {$3$};
\draw (.5,3.15) node  {$4$};
\draw (1.5,3.15) node  {$4$};
\draw (2.5,3.15) node  {$4$};
\draw (.5,4.15) node   {$5$};
\draw (1.5,4.15) node  {$5$}; 
\draw (2.5,4.15) node  {$5$}; 
\draw (-.2,.5) node  {$1$}; 
\draw (-.2,1.5) node {$1$}; 
\draw (-.2,2.5) node {$1$}; 
\draw (-.2,3.5) node  {$1$}; 
\draw (.8,.5) node  {$2$};
\draw (.8,1.5) node {$2$};
\draw (.8,2.5) node {$2$};
\draw (.8,3.5) node {$2$};
\draw (1.8,.5) node {$3$}; 
\draw (1.8,1.5) node {$3$}; 
\draw (1.8,2.5) node {$3$};
\draw (1.8,3.5) node {$3$};
\draw (2.8,.5) node  {$4$};
\draw (2.8,1.5) node {$4$}; 
\draw (2.8,2.5) node  {$4$}; 
\draw (2.8,3.5) node {$4$}; 
\draw[very thick] (0,0) -- (0,2) -- (1,2) -- (1,3) -- (2,3) -- (3,3) -- (3,4);

\end{tikzpicture}
\end{center}


Let $\mathcal{I}(m,n)$ be the set of pairs of integer sequences $(\bsy{a},\bsy{b})$ such that $\bsy{a}=(a_0,a_1,...,a_{m-1})$ satisfies $0\leq a_0 \leq a_1 \leq \cdots \leq a_{m-1} <x$ and $\bsy{b}=(b_0,b_1,..., b_{n-1})$ satisfies $0 \leq b_0 \leq b_1 \leq \cdots \leq b_{n-1} < y$.
Denote by $\mathcal{I}_{m,n}(P; \bsy{U})$ the subset of $\mathcal{I}(m, n)$ consisting of the pairs of sequences $(\bsy{a}, \bsy{b})$ that  are bounded by $P$ with respect to $\bsy{U}$. Our main result is the following theorem.

\begin{theorem} \label{main-combin} 
Assume $x, y$ are positive integers. 
The bivariate difference Gon\v{c}arov polynomial $\tilde{g}_{m,n}((x, y);\bsy{Z})$ counts the number of pairs of sequences in $\mathcal{I}(m,n)$  that are bounded by some lattice path from $O$ to $A = (m, n)$. Explicitly, we have 
$$
\frac{1}{m!n!}\tilde{g}_{m,n}((x, y);\bsy{Z})=\left|\bigcup_{P:O\to A} \mathcal{I}_{m,n}(P; \bsy{U})\right|,
$$
where $P$ ranges over all lattice paths from $O$ to $A$ which use $N$- and $E$- steps only, and the set $\bsy{U} = \{(u_{i,j} , v_{i,j} ) :0\leq i \leq m, 0 \leq j \leq n\}$ is determined by $\bsy{Z}$ according to the relations $u_{i,j} = x - x_{i,j} , v_{i,j} = y - y_{i,j}$.
\end{theorem}

Note: For the validity of the combinatorial interpretation, 
in Theorem~\ref{main-combin}  we assume that $x_{i,j}, y_{i,j} \in \mathbb{N}$, $0 \leq x_{i,j} < x$,   $0 \leq y_{i,j} < y$,
and $x_{i,j} < x_{i'j'}$, $y_{i,j} < y_{i',j}$ for all 
$(i',j') \preceq (i,j) \preceq (m,n)$. 

\begin{proof}
Our proof uses a construction similar to that in \cite[Section 6]{yan}.  For any pair of sequences $\bsy{c}=(\bsy{a},\bsy{b}) \in \mathcal{I}(m,n)$, we construct a subgraph $G(\bsy{c})$ of $D_{m,n}$ as follows:
\begin{itemize}
    \item $O=(0,0)$ is a vertex of $G(\bsy{c})$.
    \item For any vertex $(i,j)$ of $G(\bsy{c})$, 
    \begin{itemize}
        \item if $a_i < u_{i,j}$, then add the vertex $(i+1,j)$ and the $E$-step $\{(i,j),(i+1,j)\}$ to $G(\bsy{c})$.
        \item  if $b_j < v_{i,j}$, then add the vertex $(i,j+1)$ and the $N$-step $\{(i,j),(i,j+1)\}$ to $G(\bsy{c})$.
    \end{itemize}
\end{itemize}
By definition $G(\bsy{c})$ is a connected graph containing at least the vertex $O$. By Lemmas 6.3 and 6.4 of \cite{yan}, we have that if edges $\{(i,j),(i+1,j)\}$ and $\{(i,j),(i,j+1)\}$ are both in $G(\bsy{c})$, $\{(i+1,j),(i+1,j+1)\}$ and $\{(i,j+1),(i+1,j+1)\}$ are also in $G(\bsy{c})$. Furthermore, the set of vertices of $G(\bsy{c})$ has a unique maximal vertex $v(\bsy{c})$ under the order $\preceq$.

\indent Define the set $K_{m,n}(i,j) = \{\bsy{c} \in \mathcal{I}(m,n): v(\bsy{c}) = (i,j)\}$, and let $k_{m,n}(i,j) = |K_{m,n}(i,j)|$. Then $\mathcal{I}(m,n)$ is the disjoint union of all $K_{m,n}(i,j)$ for $0 \leq i \leq m$ and $0 \leq j \leq n$, and 
\[
k_{m,n}(m,n)=|K_{m,n}(m,n)| = \left|\bigcup_{P:O\to A} \mathcal{I}_{m,n}(P; \bsy{U})\right|. 
\]

Now a pair of sequences $\bsy{c} = (\bsy{a},\bsy{b})$ is in $K_{m,n}(i,j)$ if and only if there exists a lattice path $P: O \to (i,j)$ satisfying the following:
\begin{itemize}
    \item The initial segments  $\bsy{a}'=(a_0,\ldots,a_{i-1})$ and $\bsy{b}'=(b_0,\ldots,b_{j-1})$ are bounded by $P$ with respect to $\bsy{U}$.  That is,  $(\bsy{a}',\bsy{b}')$ is in $K_{i,j}(i,j)$.  
     There are $k_{i,j}(i,j)$ such pairs of initial segments. 
    \item The integer sequence $(a_i,\ldots,a_{m-1})$ 
    satisfies $u_{i,j} \leq a_i \leq \cdots \leq a_{m-1} \leq x-1$. 
    \item The integer sequence $(b_j,\ldots,b_{n-1})$ 
    satisfies $v_{i,j} \leq b_j \leq \cdots \leq b_{n-1} \leq y-1$. 
\end{itemize}
Thus,
\begin{align*}
    k_{m,n}(i,j) &= \#\{\bsy{c} = (\bsy{a},\bsy{b}): \bsy{a}',\bsy{b}' \textnormal{ satisfy the three above conditions} \} \\
    &= k_{i,j}(i,j) \binom{x-1-u_{i,j}+m-i}{m-i}
    \binom{y-1-v_{i,j}+n-j}{n-j} \\
    &= k_{i,j}(i,j) \frac{(x-u_{i,j})^{(m-i)}}{(m-i)!}\frac{(y-v_{i,j})^{(n-j)}}{(n-j)!}.
\end{align*}
Hence,
$$ \frac{x^{(m)}}{m!}\frac{y^{(n)}}{n!} 
    = |\mathcal{I}(m,n)|
    = \sum_{i=0}^m \sum_{j=0}^n k_{m,n}(i,j) 
    = \sum_{i=0}^m \sum_{j=0}^n \frac{(x-u_{i,j})^{(m-i)}}{(m-i)!}\frac{(y-v_{i,j})^{(n-j)}}{(n-j)!} k_{i,j}(i,j), $$
or
$$   x^{(m)}y^{(n)} = \sum_{i=0}^m\sum_{j=0}^n 
\binom{m}{i}\binom{n}{j} (x-u_{i,j})^{(m-i)}(y-v_{i,j})^{(n-j)} i!j!k_{i,j}(i,j). $$
Comparing this to the linear recursion formula \eqref{linear} and 
using the initial values $\tilde{g}_{0,0}((x,y);\bsy{Z})=k_{0,0}(0,0)=1$, 
we conclude that $\tilde{g}_{m,n}((x,y);\bsy{Z}) = m!n!k_{m,n}(m,n)$, where $z_{i,j}=(x_{i,j}, y_{i,j})$ with $
x_{i,j}=x-u_{i,j}$ and $y_{i,j}=y-v_{i,j}$. 
\end{proof}

\begin{corollary}\label{cor-ipf} 
Under the same assumptions of Theorem~\ref{main-combin}, we have
\[
\left|\bigcup_{P:O\to A} \mathcal{I}_{m,n}(P; \bsy{U})\right|
= \frac{1}{m!n!} \tilde{g}_{m,n}((0,0); -\bsy{U}). 
\]
\end{corollary}
\begin{proof} 
It follows from Theorem \ref{shift}, the shift-invariant formula, and the  relation $\bsy{Z}=(x,y)-\bsy{U}$. 
\end{proof} 

Recall that 
for a sequence of real numbers $\mathbf{x}=(x_1, x_2, \dots, x_n)$,
the $i$-th order statistic, $x_{(i)}$, is the $i$-th term in the non-decreasing rearrangement   $x_{(1)} \leq x_{(2)} \leq \cdots \leq x_{(n)}$  of $\mathbf{x}$. 
In \cite{yan}, a generalized notion of \emph{2-dimensional parking functions}  was introduced in terms of order statistic constraints on a double sequence. 

\begin{definition} \label{U-parking} 
Given a set of nodes $\bsy{U} =  \{(u_{i,j},v_{i,j}) : i,j \in \mathbb{N} \} \subset \mathbb{N}^2$ satisfying $u_{i,j} \leq u_{i',j'}$ and $v_{i,j} \leq v_{i',j'}$ when $(i,j)\preceq (i',j')$, a pair of non-negative integer sequences $(\bsy{a},\bsy{b})$ with $\bsy{a}=(a_0, a_1, \ldots ,a_{m-1})$ and $\bsy{b}=(b_0, b_1, \ldots, b_{n-1})$ is a \textit{2-dimensional $\bsy{U}$-parking function} if and only if the order statistics of $(\bsy{a},\bsy{b})$ are bounded by some lattice path from the origin to $(m,n)$ with respect to $\bsy{U}$.
\end{definition}

Clearly the set $K_{m,n}(m,n)$ consists of those 2-dimensional $\bsy{U}$-parking functions with non-decreasing sequences $\bsy{a}$ and $\bsy{b}$.  Following the convention in the univariate case, elements in $K_{m,n}(m,n)$ are called \emph{2-dimensional increasing $\bsy{U}$-parking functions}, and 
we replace the notation $K_{m,n}(m,n)$ with  $\mathcal{IPF}^{(2)}_{m,n}( \bsy{U})$. 
Hence Corollary \ref{cor-ipf} gives a formula that enumerates the set  $\mathcal{IPF}^{(2)}_{m,n}( \bsy{U})$. 

If there exist some sequences $\bsy{\alpha} = (\alpha_0, \alpha_1, \dots)$ and 
$\bsy{\beta} = (\beta_0, \beta_1, \dots)$ such that
$(u_{i,j}, v_{i,j}) = (\alpha_i, \beta_j)$, then 
\begin{align*}
\frac{1}{m!n!} \tilde{g}_{m,n}((0,0); -\bsy{U}) = \frac{1}{m!}
\tilde{g}_m(0; -\bsy{\alpha}) 
\cdot \frac{1}{n!} \tilde{g}_n (0; -\bsy{\beta}). 
\end{align*}
In this case $\mathcal{IPF}^{(2)}_{m,n}(U)$ is the direct product of 
the set of non-decreasing integer sequences of length $m$ bounded by $\bsy{\alpha}$  and the set of non-decreasing integer sequences of length $n$ bounded by $\bsy{\beta}$

A more  interesting case is  when the node-set  $\bsy{U}$ is obtained from $\mathbb{N}^2$ by an affine transformation, i.e., there is a $2 \times 2$ matrix $A$ such that 
\begin{equation} \label{affine} 
\left[\begin{array}{c} 
u_{i,j} \\ 
v_{i,j} 
\end{array} \right] = A 
\left[\begin{array}{c} 
i \\ 
j
\end{array} \right] + 
\left[\begin{array}{c} 
s \\ 
t
\end{array} \right]. 
\end{equation} 
\vanish{
When $s=t=0$ in \eqref{affine}, the corresponding bivariate difference Gon\v{c}arov polynomials are called \emph{delta Abel polynomials}, which are analogs of the Abel polynomial $A(x)=x(x-a)^{n-1}$ and satisfy a 
 multivariate identity of  binomial type. 
 In \cite[Theorem 5.1]{tringali}, a closed formula is given for general multivariate delta Abel polynomials. Specializing  that formula to difference operators, we obtain the following result.  
 }

\begin{proposition} \label{pro-affine} 
Let $\bsy{U}$ be given by \eqref{affine} and 
\[
A=\left[ \begin{array}{cc} 
  a & b  \\
  c & d 
  \end{array} 
  \right],
  \]
  with $a, b, c, d, s, t\in \mathbb{N}$. 
Then 
\begin{align*}
\left| \mathcal{IPF}_{m,n}^{(2)}(\bsy{U}) \right|
= \frac{1}{m!n!}(st+bnt+scm)(s+am+bn+1)^{(m-1)} (t+cm+dn+1)^{(n-1)}.
\end{align*}
\end{proposition} 
\begin{proof} 
Using the shift invariant formula, we have
$$\tilde{g}_{m,n}((0,0); -\bsy{U} )
=\tilde{g}_{m,n}((s,t); \bsy{Z}),
$$
where the grid $\bsy{Z}$ has nodes $\{z_{i,j}=(x_{i,j}, y_{i,j}): (0,0) \preceq (i,j)\preceq (m,n)\}$ given by 
$x_{i,j} =-ai-bj$ and  $y_{i,j}=-ci-dj$. 
Using Equation \eqref{Abel-2} in Theorem \ref{Abel-thm} 
we obtain 
\[
\tilde{g}_{m,n}((x,y); Z) = \\
(xy+ xy_{m,0}+y x_{0,n}) (x+x_{m,n}+1)^{(m-1)}(y+y_{m,n}+1)^{(n-1)}, 
\]
which leads to the desired formula when substituted with $x=s$ and $y=t$.
\end{proof} 

\begin{corollary} \label{cor-affine} 
Let $\bsy{U}$ be given by 
\begin{equation*}
  \left[\begin{array}{c} 
u_{i,j} \\ 
v_{i,j} 
\end{array} \right] = 
\left[ \begin{array}{cc} 
  0 & b  \\
  c & 0 
  \end{array} 
  \right]
\left[\begin{array}{c} 
i \\ 
j
\end{array} \right] + 
\left[\begin{array}{c} 
1 \\ 
1
\end{array} \right]. 
\end{equation*} 

Then 
\begin{align*}
\left| \mathcal{IPF}_{m,n}^{(2)}(\bsy{U}) \right|
=\frac{1+bn+cm}{(1+bn)(1+cm)} \binom{bn+m}{m} \binom{cm+n}{n}.
\end{align*}
\end{corollary} 

In particular, when $b=c=1$ in Corollary \ref{cor-affine}, the 2-dimensional increasing $\bsy{U}$-parking functions 
coincide with the \emph{increasing $(p,q)$-parking functions} defined by Cori and Poulalhon \cite{cori}, and Corollary \ref{cor-affine} gives the Narayana number 
\[
\frac{1+m+n}{(1+m)(1+n)} \binom{m+n}{m}\binom{m+n}{n}
= \frac{1}{1+m+n} \binom{1+m+n}{m} \binom{1+m+n}{n},
\]
agreeing with Proposition 14 of \cite{cori}. 


\section{Multivariate Cases} 

\indent 
Let $d$ be a fixed integer $\geq 1$. For a vector $\mathbf{v} \in \mathbb{F}^d$, we denote by $v_j$ the $j$-th component of $\mathbf{v}$. Given $\mathbf{n}=(n_1, \dots, n_d) \in \mathbb{N}^d$, we set $\mathbf{n}!=n_1!n_2!\cdots n_d!$. For $\mathbf{k}, \mathbf{n} \in \mathbb{N}^d$, $
\mathbf{k} \preceq \mathbf{n}$ means $k_i \leq n_i$ for all $1 \leq i \leq d$, and $\binom{\mathbf{n}}{\mathbf{k}} = \binom{n_1}{k_1} \cdots \binom{n_d}{k_d}$. With such notation in place, we can define the $d$-dimensional difference Gon\v{c}arov polynomials 
with respect to the system of difference operators $(\Delta_{x_1}, \dots, \Delta_{x_d})$. Given a grid 
$\bsy{Z}=\{ z_{\mathbf{k}} \in \mathbb{F}^d: \mathbf{k} \in \mathbb{N}^d\}$, there is a unique polynomial $t_{\mathbf{n}}(\mathbf{x};\bsy{Z})$ of coordinate degree $\mathbf{n} \in \mathbb{N}^d$ satisfying 
\[
\varepsilon(z_{\mathbf{k}}) \Delta_{x_1}^{k_1} \cdots 
\Delta_{x_d}^{k_d} (t_{\mathbf{n}}(\mathbf{x}; \bsy{Z})) = \mathbf{n}! \delta_{\mathbf{k}, \mathbf{n}} 
\]
for all $\mathbf{k} \preceq \mathbf{n}$. This polynomial 
is the multivariate difference Gon\v{c}arov polynomial indexed by $\mathbf{n}$, which we will denote by $\tilde{g}_{\mathbf{n}}(\mathbf{x};\bsy{Z})$.  Theorems 2--9 can all be extended  straightforwardly to $d$ dimensions, where the summations are over the set  $\{\mathbf{k} \in \mathbb{N}^d: \mathbf{k} \preceq \mathbf{n}\}$, and the binomial coefficient $\binom{m}{i} \binom{n}{j}$ is replaced with $\binom{\mathbf{n}}{\mathbf{k}}$.

To generalize all definitions and results of Section 4 to $d$ dimensions for $d > 2$, we must first define the $d$-dimensional analogs of the sets $\mathcal{I}(m,n)$ and $\mathcal{I}_{m,n}(P;\bsy{U})$. Let $\mathbf{n}=(n_1,\ldots,n_d) \in \mathbb{N}^d$, and fix a $d$-dimensional set of weight-vectors
$$ \bsy{U} = \{ u_{\mathbf{k}} \in \mathbb{N}^d : \mathbf{k} \in \mathbb{N}^d, u_{\mathbf{k},i} \leq u_{\mathbf{k}',i} \text{ whenever } \mathbf{k} \preceq \mathbf{k}' \text{ and } 1 \leq i \leq d \}, $$
where $u_{\mathbf{k},i}$ is the $i$-th entry of the point $u_{\mathbf{k}}$. Extend the weighted directed graph $D_{m,n}$ to $d$ dimensions by taking as vertices the points  $\{(k_1,\ldots,k_d): 0 \leq k_i \leq n_i \text{ for all } i=1,\ldots,d\}$ and as edges all steps $\mathbf{e}_i$, $1 \leq i \leq d$, where $\mathbf{e}_i = (0,\ldots,0,1,0,\ldots,0)$ has 1 in the $i$-th entry. An edge $\ell$ from $\mathbf{k}=(k_1,\ldots,k_i,\ldots,k_d)$ to $(k_1,\ldots,k_i+1,\ldots,k_d)$ is assigned the weight $wt(\ell) = u_{\mathbf{k},i}$. 

\indent Suppose $P = \ell_1 \ell_2 \ldots \ell_n$ is any lattice path from the origin $O = (0,\ldots,0)$ to the point $(n_1,\ldots,n_d)$, where each $\ell_i \in \{\mathbf{e}_j: 1 \leq j \leq d\}$ and $n=n_1+\cdots+n_d$. Consider a $d$-tuple of non-decreasing integer sequences $(\bsy{a}^{(1)},\ldots,\bsy{a}^{(d)})$, where $\bsy{a}^{(i)}=(a^{(i)}_0,a^{(i)}_1,\ldots,a^{(i)}_{n_i-1})$ for $1 \leq i \leq d$. 
Then we say that $(\bsy{a}^{(1)},\ldots,\bsy{a}^{(d)})$ is bounded by $P$ with respect to the set $\bsy{U}$ if the following condition is satisfied for each $r = 1,2,\ldots,n$: 
$$ a^{(i)}_{k_i} < u_{\mathbf{k},i} \text{ if } \ell_r \text{ is an } \mathbf{e}_i \text{--step  from } \bsy{k}=(k_1,\ldots,k_i,\ldots,k_d) \text{ to } (k_1,\ldots,k_i+1,\ldots,k_d).$$

For a fixed $\mathbf{x}=(x_1,\ldots,x_d) \in \mathbb{N}^d$, let the set $\mathcal{I}(\mathbf{n})$ consist of all $d$-tuples of non-decreasing integer sequences $(\bsy{a}^{(1)},\ldots,\bsy{a}^{(d)}) \in \mathbb{N}^{n_1} \times \cdots \mathbb{N}^{n_d}$ such that $0\leq \bsy{a}_j^{(i)} < x_i$ for all $1\leq i \leq d$ and $0 \leq j < n_i$,
and let $\mathcal{I}_{\mathbf{n}}(P;\bsy{U})$ be the subset of $\mathcal{I}(\mathbf{n})$ containing all $d$-tuples that are bounded by $P$ with respect to $\bsy{U}$. The following is the $d$-dimensional analog of Theorem 11 and Corollary 12.

\begin{theorem} \label{d-sequence} 
Let $\mathbf{x}=(x_1,\ldots,x_d), \mathbf{n}=(n_1,\ldots,n_d) \in \mathbb{N}^d$ and $\bsy{Z} = \{z_{\mathbf{i}} : \mathbf{i} \in \mathbb{N}^d \} \subset \mathbb{N}^d$. The multivariate difference Gon\v{c}arov polynomial $\tilde{g}_{\mathbf{n}}(\mathbf{x};\bsy{Z})$ gives the number of $d$-tuples of sequences in $\mathcal{I}(\mathbf{n})$ that are bounded by some lattice path from the origin to $A = (n_1,\ldots,n_d)$. In particular,
\begin{equation}\label{d-dim union}
\Bigg| \bigcup_{P:O \to A} \mathcal{I}_{\mathbf{n}}(P;U) \Bigg| = \frac{1}{\mathbf{n}!} \tilde{g}_{\mathbf{n}}(\mathbf{x};\bsy{Z}) = \frac{1}{\mathbf{n}!} \tilde{g}_{\mathbf{n}}((0,\ldots,0);-\bsy{U}),
\end{equation}
where the set $\bsy{U} = \{u_{\mathbf{k}} \in \mathbb{N}^d : \mathbf{k} \preceq \mathbf{n}\}$ is defined by $u_{\mathbf{k},i} = x_i - z_{\mathbf{k},i}$.
\end{theorem}

\indent As in the 2-dimensional case, we can define a $d$-dimensional $\bsy{U}$-parking function according to certain constraints imposed on the order statistics of a $d$-tuple.

\begin{definition}
Let $d > 2$ be an integer. 
Given a set of nodes $\bsy{U} = \{ u_{\mathbf{k}} \in \mathbb{F}^d : \mathbf{k} \in \mathbb{N}^d\}$ such that $u_{\mathbf{k},i} \leq u_{\mathbf{k}',i}$ whenever $\mathbf{k} \preceq \mathbf{k}'$ and $1 \leq i \leq d$, a $d$-tuple of integer sequences $(\bsy{a}^{(1)},\ldots,\bsy{a}^{(d)}) \in \mathbb{N}^{n_1} \times \cdots \times \mathbb{N}^{n_d}$ is a $d$-dimensional $\bsy{U}$-parking function if and only if the order statistics of $(\bsy{a}^{(1)},\ldots,\bsy{a}^{(d)})$ are bounded by some lattice path from the origin to $(n_1,\ldots,n_d)$ with respect to $\bsy{U}$.
\end{definition}

Of particular importance to us are the \textit{d-dimensional increasing $\bsy{U}$-parking functions} $(\bsy{a}^{(1)},\ldots,\bsy{a}^{(d)})$, which have non-decreasing constituent sequences $\bsy{a}^{(1)},\ldots,\bsy{a}^{(d)}$. Using notation consistent with the bivariate case, we will denote the set of all $d$-dimensional increasing $\bsy{U}$-parking functions by $\mathcal{IPF}^{(d)}_{\mathbf{n}}(\bsy{U})$. Note then that the set $\mathcal{IPF}^{(d)}_{\mathbf{n}}(\bsy{U})$ is equivalent to the union of the sets $\mathcal{I}_{\mathbf{n}}(P;U)$ over all lattice paths $P$ from the origin to $A=(n_1,\ldots,n_d)$, so that Theorem 15 also yields a formula for $\# \mathcal{IPF}^{(d)}_{\mathbf{n}}(\bsy{U})$ in terms of a multivariate difference Gon\v{c}arov polynomial.


In the following we  enumerate the set 
$\mathcal{IPF}^{(d)}_{\mathbf{n}}(\bsy{U})$ when the node-set 
$\bsy{U}$ is affine, meaning there exists a $d \times d$ matrix $A$ and vector $\mathbf{s} = (s_1,\ldots,s_d) \in \mathbb{F}^d$ such that $u_{\mathbf{k}} = A \mathbf{k} +\mathbf{s}$ for all $\mathbf{k} \in \mathbb{N}^d$, where $\mathbf{k}$, $u_{\mathbf{k}}$,  and $\mathbf{s}$ are treated as column vectors.
We express such an affine node-set $\bsy{U}$ by $\bsy{U} = A \mathbb{N}^d + \mathbf{s}$.  
Increasing $\bsy{U}$-parking functions associated to affine $\bsy{U}$  relate to the notion of $(p_1, p_2, \dots, p_d)$-parking functions in  \cite{cori} and the notion of $G$-parking functions \cite{Postnikov}, when
$G$ is a complete $d$-partite graph with a distinguished root. 

We use a closed formula for $d$-dimensional delta Abel polynomials proved in \cite[Theorem 6.1]{tringali}. The following statement is specialized to the system of operators
$(\Delta_{x_1}, \ldots, \Delta_{x_d})$. 

\begin{theorem} \cite{tringali}  \label{d-abel} 
Let $\mathbf{x}$, $\mathbf{n}$, and $\bsy{Z}$ be as in Theorem 15, and suppose $\bsy{Z} = A \mathbb{N}^d$ for some $d \times d$ matrix $A = (a_{i,j})$. 
Let $B = (b_{i,j})$ be the $d \times d$ diagonal matrix defined by $b_{i,i} = x_i - z_{\mathbf{n},i}$, and let $C = (c_{i,j})$ be the $d \times d$ matrix defined by $c_{i,j} = z_{n_i \mathbf{e}_i,j}$. Then
$$ \tilde g_{\mathbf{n}}(\mathbf{x}; \bsy{Z}) =  \det(B + C) \prod_{i=1}^d (x_i - z_{\mathbf{n},i} + 1)^{(n_i-1)}.$$
\end{theorem} 


Theorem \ref{d-abel} 
yields the following result on the enumeration of $d$-dimensional increasing $\boldsymbol{U}$-parking functions.

\begin{proposition}
Let $\mathbf{x}$ and  $\mathbf{n}$ be as in Theorem 15, and suppose $\bsy{U} = A \mathbb{N}^d + \mathbf{s}$ for some $d \times d$ matrix $A = (a_{i,j})$ and $\mathbf{s} = (s_1,\ldots,s_d) \in \mathbb{Z}^d$. Let $B = (b_{i,j})$ be the $d \times d$ diagonal matrix defined by $b_{i,i} = x_i + z_{\mathbf{n},i}$, and let $C = (c_{i,j})$ be the $d \times d$ matrix defined by $c_{i,j} = -z_{n_i \mathbf{e}_i,j}$, where $\bsy{Z}=A \mathbb{N}^d$. Then
$$ \# \mathcal{IPF}^{(d)}_{\mathbf{n}}(\bsy{U}) = \frac{1}{\mathbf{n}!} \det(B + C) \prod_{i=1}^d (s_i + z_{\mathbf{n},i} + 1)^{(n_i-1)}.$$
\end{proposition}
\begin{proof}
From Theorem \ref{d-sequence} and the shift-invariant property of Gon\v{c}arov polynomials, we have 
\begin{equation*}
    \# \mathcal{IPF}^{(d)}_{\mathbf{n}}(\bsy{U}) = 
    \frac{1}{\mathbf{n}!}  \tilde{g}_{\mathbf{n}}(\bsy{0};-\bsy{U}) 
    = \frac{1}{\mathbf{n}!}  \tilde{g}_{\mathbf{n}}(\mathbf{s}; -\bsy{Z} ).
\end{equation*}
Then we use Theorem 16 and notice that the transition matrix for the grid $-\bsy{Z}$ is  $-A$. 
\end{proof}

\begin{corollary}
Suppose $\bsy{U} = A \mathbb{N}^d + \mathbf{s}$, where $A = (a_{i,j})$ is a $d \times d$ matrix with
\[
a_{i,j} = \left\{ \begin{array}{ll} 
 \alpha_j & \text{ if  } i \neq j \\ 
 0    & \text{ i=j} 
 \end{array} \right. 
\]
and $\mathbf{s} = (s_1,\ldots,s_d) \in \mathbb{Z}^d$. Then we have
$$ \# \mathcal{IPF}^{(d)}_{\mathbf{n}}(\bsy{U}) = \frac{1}{\mathbf{n}!} \Bigg( 1 - \sum_{j=1}^d \frac{\alpha_j n_j}{s_j + N} \Bigg) \prod_{i=1}^d (s_i + N)(s_i + N - \alpha_i n_i + 1)^{(n_i-1)}, $$
where $N = \sum_{i=1}^d \alpha_i n_i$. 
\end{corollary}
\begin{proof} 
The result follows from Proposition 17 and the computation of $\det(B+C)$, where $B$ and $C$ are the $d \times d$ matrices defined in Proposition 17. We have
\begin{align*}
B+C &= 
\left[\begin{array}{cccc} 
s_1 + N - \alpha_1 n_1 & -\alpha_1 n_1 & \cdots & -\alpha_1 n_1 \\ 
-\alpha_2 n_2 & s_2 + N - \alpha_2 n_2 & \cdots & -\alpha_2 n_2 \\
\vdots & \vdots & \ddots & \vdots \\
-\alpha_d n_d & -\alpha_d n_d & \cdots & s_d + N - \alpha_d n_d \\
\end{array} \right]
\end{align*}
Subtracting  column $j-1$ from column $j$ for $j=d, d-1, \dots, 2$, we get 
\begin{align*} 
\det(B+C) &= 
 \left|\begin{array}{cccccc}
s_1 + N - \alpha_1 n_1 & -s_1 - N & 0 & \cdots & 0 & 0 \\
-\alpha_2 n_2 & s_2 + N & -s_2 - N & \cdots & 0 & 0 \\
-\alpha_3 n_3 & 0 & s_3 + N & \cdots & 0 & 0 \\
\vdots & \vdots & \vdots & \ddots & \vdots & \vdots \\
-\alpha_{d-1} n_{d-1} & 0 & 0 & \cdots & s_{d-1} + N & -s_{d-1} - N \\
-\alpha_d n_d & 0 & 0 & \cdots & 0 & s_d + N \\
\end{array}\right| \\
&= (s_1 + N - \alpha_1 n_1) \prod_{i=2}^d (s_i + N) + \sum_{i=2}^d (-1)^{1+i}(-\alpha_i n_i) \bigg[ \prod_{j=1}^{i-1}(-s_j - N) \bigg] \bigg[ \prod_{j=i+1}^d (s_j + N) \bigg] \\
&= \Bigg( 1 - \sum_{j=1}^d \frac{\alpha_j n_j}{s_j + N} \Bigg) \prod_{i=1}^d (s_i + N),
\end{align*}
where $N = \sum_{i=1}^d \alpha_i n_i$. Also, note that $z_{\mathbf{n},i} = u_{\mathbf{n},i} - s_i = N - \alpha_i n_i$.
\end{proof} 

When $\alpha_i = s_i = 1$ for all $1 \leq i \leq d$ in Corollary 18,  the second and the third authors showed \cite{Snider} that the set of $d$-dimensional increasing $\bsy{U}$-parking functions is precisely the set of increasing $(p_1,p_2,\ldots,p_d)$-parking functions described by Cori and Poulalhon \cite{cori}. 
 According to \cite[Proposition 19]{cori}, the latter are counted by the formula
$$ \frac{1}{N+1} \prod_{i=1}^d \binom{N+1}{n_i}, $$
which matches  the value for $\# \mathcal{IPF}^{(d)}_{\mathbf{n}}(\bsy{U})$ given by Corollary 18 after simple algebraic manipulation.


\end{document}